\documentclass[10pt]{amsart}
\usepackage{amsmath,amssymb,amsthm}
\usepackage{graphicx}
\usepackage{color}
\newtheorem{theorem}{Theorem}[section]    
\newtheorem{lemma}[theorem]{Lemma}
\newtheorem{claim}{Claim} 
\newtheorem{proposition}[theorem]{Proposition}

\newtheorem{corollary}[theorem]{Corollary}
\theoremstyle{definition}
\newtheorem{definition}[theorem]{Definition}
\newtheorem{example}[theorem]{Example}

\newtheorem*{remark*}{Remark}

\newcommand{\R}{\mathbb{R}}

\newcommand{\Crossing}{\!\!
\raisebox{-12pt}{
\begin{picture}(24,28)
\put(0,2){\line(1,1){24}}
\put(24,2){\line(-1,1){24}}
\end{picture} } 
}

\newcommand{\Smooth}{\!\!
\raisebox{-12pt}{
\begin{picture}(24,28)
\qbezier(0,2)(14,14)(0,26)
\qbezier(24,2)(10,14)(24,26)
\end{picture}} 
}

\newcommand{\Smoothh}{\!\!
\raisebox{-12pt}{
\begin{picture}(24,28)
\qbezier(0,2)(12,16)(24,2)
\qbezier(0,26)(12,14)(24,26)
\end{picture}} 
}

\begin{document}

\title[Crossing number of links on surface]{On the crossing number of knots and links on surfaces in 3-manifolds}
\author{Tetsuya Ito}
\address{Department of Mathematics, Kyoto University, Kyoto 606-8502, JAPAN}
\email{tetitoh@math.kyoto-u.ac.jp}

\maketitle

\begin{abstract}
We give a lower bound of the minimum crossing number of knots and links projected on a 2-sided surface in a 3-manifold using the rank of fundamental groups. This shows that the crossing number can take arbitrarily large values. In contrast, we show that for a branched surface, the minimum crossing number is always zero. As an application, we give an upper bound of the three-page index in terms of the crossing number.
\end{abstract}

\section{Introduction}

Let $M$ be a compact 3-manifold, possibly non-orientable or having a non-empty boundary, and let $F$ be a properly embedded connected surface possibly non-orientable which is 2-sided. Namely, there is an embedding $h:F\times[-1,1] \rightarrow M$ such that $h(p,0)=p$ for all $p \in F \subset M$ and $h(\partial F \times [-1,1]) = \partial h(F \times [-1,1]) \cap \partial M$. 

We say that a link $L$ in $M$ is \emph{projectible on $F$} if one can put $L$ so that it sits in $F\times[-1,1] \subset M$.
For a projectible link $L$, an \emph{$F$-diagram} $D$ of $L$ is the image of the projection $\pi:F\times [-1,1] \rightarrow F$ of the link $L$ put in $F \times [-1,1]$, taken so that it has only transverse double point singularities. 

By a \emph{crossing} and a \emph{region} of a diagram $D$, we mean a double point of $D$ and a connected component of $F \setminus D$, respectively. We denote by $c(D)$ and $R(D)$ the number of crossings and regions of $D$, respectively. 

The (minimum) \emph{$F$-crossing number} of a link $L$ in $M$ is defined by
\[ c(L;F\subset M):=\begin{cases}
\min\{ c(D)\:|\: D \mbox{ is an }F\mbox{-diagram of }L\} & \mbox{If }L \mbox{ is projectible on }F \\
\infty & \mbox{Otherwise} \\
 \end{cases}
\]
When $M=S^{3}$ we often write $c(L;F)$ instead of $c(L;F\subset S^{3})$. In particular, when $F=S^{2} \subset S^{3}$ then $c(L;S^2)$ is nothing but the usual crossing number $c(L)$ of a link $L$.

In \cite{oz} Ozawa showed the inequality 
\[ c(K;F)\geq 2(t(K)-\delta(F))+1 \]
for a knot $K$ in $S^{3}$ with $t(K)> \delta(F)$ and a closed surface $F$ in $S^{3}$. Here $t(K)$ is the tunnel number and $\delta(F)=g(M_1;F)+g(M_2;F)-g(F)$ is the Heegaard deficiency, where $g(M_1;F)$ and $g(M_2;F)$ are the relative Heegaard genera of the connected components $M_1$ and $M_2$ of $M \setminus F$, and $g(F)$ is the genus of $F$.

As a consequence, he showed that $c(K;F)$ is unbounded for any surface $F$. As mentioned in \cite{oz}, the existence of knots with a large $F$-crossing number is not obvious as there are many \emph{branched} surfaces supporting all links (such as Ghrist's universal template \cite{gh}).
That is, every knot can be put on the branched surface as a simple closed curve. 

The aim of this note is to give an inequality of the $F$-crossing number based on the fundamental groups and elementary combinatorics of diagrams. Our inequality uses the \emph{rank} $r(G)$ of a finitely generated group $G$, the minimum number of generators of $G$. We define 
\[ r(F \subset M)= \begin{cases}
r(\pi_1(M_+)) + r(\pi_1(M_-)) & \mbox{(If } F \mbox{ is separating)} \\
r(\pi_1(M \setminus F))+1 & \mbox{(If } F \mbox{ is non-separating)}
\end{cases}\]
where $M_{\pm}$ are the connected components of $M \setminus F$.

We denote by $G(L)$ the link group of $L$, the fundamental group of the link exterior $E(L) = M \setminus N(L)$ where $N(L)$ is the open tubular neighborhood of $L$.
We will introduce a notion of \emph{generating number $x(D)$} of an $F$-diagram $D$ in Section \ref{section:generating}.

Using the generating number we establish the following inequality.

\begin{theorem}
\label{theorem:main} For an $F$-diagram $D$ of a link $L$ in $M$, 
\[ r(G(L)) \leq r(F \subset M)+ x(D) -1 \]
holds.
\end{theorem}

In Lemma \ref{lemma:region-crossing}, we show that $x(D) \leq c(D) + \#L +1$ where $\#L$ is the number of connected components of the link $L$ so we get the following lower bound of the $F$-crossing number.

\begin{corollary}
\label{cor:crossing}
$c(L;F\subset M) \geq r(G(L))- r(F \subset M)-\# L$.
\end{corollary}

Consequently we get the infiniteness of $F$-crossing number.

\begin{corollary}
\label{cor:unbounded}
For every 2-sided surface $F$ in $M$ and $n >0$, there exists a projectible knot $K$ such that $c(K;F\subset M)>n$.
\end{corollary}

\begin{proof}
By Corollary \ref{cor:crossing}, it is sufficient to show that for each $N>0$ there is a projectible knot whose knot group has rank $\geq N$. Take a knot $K_N$ in the 3-ball $B^3$ with $r(\pi_1(B^3 \setminus K_N)) \geq N$, such as a suitable Montesinos knot \cite[Theorem 0.1]{lm}. 
We regard $K_N$ as a knot in $M$ by embedding $B^{3}$. Clearly $K_N$ is projectible on $F$. Since $G(K) = \pi_1(M) \ast \pi_1(B^{3} \setminus K)$, by Grushko's Theorem, 
\[ r(G(K)) = r( \pi_1(M)) + r(\pi_1(B^{3} \setminus K) \geq N. \]
\end{proof}

The advantage of using the fundamental group is that one can use almost the same argument for more general situations which are left as further questions in \cite{oz}. Furthermore, since the generating number $x(D)$ reflects the property of diagrams viewed as a graph on a surface, one can use Theorem \ref{theorem:main} to give a constraint for a knot to admit $F$-diagrams of certain types. In particular, our argument gives a unified and general version of various estimates of diagrammatic quantities in terms of rank of knot groups.

We will also discuss knot diagrams on a branched surface in Section \ref{section:branch}. We show that the $F$-crossing of projectible link is always zero for branched surface (Theorem \ref{theorem:appendix}) which makes a sharp contrast with Corollary \ref{cor:unbounded}. As an application, we prove the following inequality of three-page index $\alpha_3(L)$ (see Definition \ref{definition:3-page}) that answers \cite[Conjecture 1]{yo} affirmatively.

\begin{theorem}\label{theorem:3-page-inequality}
For a non-split link $L$ in $S^{3}$, $\alpha_3(L) \leq \frac{23}{8}c(L) + \frac{3}{4}$.
\end{theorem}

\section{Crossing number of $F$-diagrams}

\subsection{Generating number and Proof of Theorem \ref{theorem:main}}
\label{section:generating}

For an $F$-diagram $D$ of a link $L$, let $\mathcal{R}(D)$ be the set of its regions. We call the four open quadrants around a crossing $c$ of $D$ the \emph{corner} around $c$. We say that the corner $x$ \emph{belongs to a region $R$} if $x \subset R$. More generally, for a subset $X$ of $\mathcal{R}(D)$, we say that a corner $x$ belongs to $X$ if $x$ belongs to a region $R \in X$.

\begin{definition}[Region derived from other regions]
\label{definition:derived}
Let $X$ be a subset of $\mathcal{R}(D)$. We say that a region $R$ is \emph{derived from $X$} if there exists a crossing $c$ such that for the four corners around $c$, exactly one corner belongs to the region $R$ and that the other three corners belong to $X$.
\end{definition}

We denote by 
\[ \mathcal{D}(X)=X \cup \{\mbox{regions that are derived from }X\}.\]
For $i=0,1,\ldots,$ we inductively define 
\[ X=\mathcal{D}^{0}(X) \subset \mathcal{D}^{1}(X) \subset \cdots \subset \mathcal{D}^{i}(X) \subset \mathcal{D}^{i+1}(X) \subset \cdots \]
by $\mathcal{D}^{i}(X) = \mathcal{D}(\mathcal{D}^{i-1}(X))$ for $i>0$. We put 
\[ \mathcal{D}^{\infty}(X) = \bigcup_{i=0}^{\infty} \mathcal{D}^{i}(X). \]
Since $\mathcal{R}(D)$ is a finite set, $\mathcal{D}^{\infty}(X)=\mathcal{D}^{N}(X)$ for sufficiently large $N$.

\begin{definition}[Generating number]
We say that a set of regions $X \subset \mathcal{R}(D)$ is \emph{generating} if $\mathcal{D}^{\infty}(X)=\mathcal{R}(D)$. We define the \emph{generating number} $x(D)$ of $D$ by 
\[ x(D) = \min \{\#X \: | \: X \subset \mathcal{R}(D) \mbox{ is generating} \}\]
\end{definition}

We denote by $\#D$ the number of connected components of $D$ as a subset of $F$.
As we have mentioned in the introduction, we have the following obvious bounds.

\begin{lemma}
\label{lemma:region-crossing}
For an $F$-diagram $D$, $x(D) \leq R(D) \leq c(D) +  \#D + 1$ holds.
\end{lemma}
\begin{proof}
The first inequality is obvious from the definition of $x(D)$. 
If $c(D)=0$, $R(D) \leq \#D + 1$. 
For a diagram $D$ with $c(D)>0$ we resolve a crossing 
\Crossing to either \Smooth or \Smoothh so that $\#D$ is unchanged. Each resolution reduces the number of regions by at most one so $R(D) \leq c(D)+\# D +1 $.
\end{proof}

This bound is sharp as the following example shows.
\begin{example}
Let $D$ be an $F$-diagram that consists of $m$ parallel disjoint separating circles. Then $R(D)=m+1$. On the other hand, since $c(D)=0$, for any subset $X$ of $\mathcal{R}(D)$, $\mathcal{D}^{\infty}(X)=\mathcal{D}(X)=X$. Thus $x(D)=R(D)=m+1$.
\end{example}

\begin{proof}[Proof of Theorem \ref{theorem:main}]
Let $X=\{R_1,\ldots, R_{x(D)}\} \subset \mathcal{R}(D)$ be a generating subset of regions of $D$.

Let $M_F = M \setminus (F\times(-1,1))$ and $F_\pm = F\times\{\pm 1\} \subset M_F$.
We put $F' = F-N(D) \subset F$.
For a region $R$ of $D$, we denote by $R_{\pm}$ the subset $(R \cap F') \times \{\pm 1\} \subset F_{\pm} \subset M_{F}$. Let $M_1$ be the 3-manifold obtained from $M_F$ by (re-)gluing $(R_1)_{\pm}$. If $F$ is separating, $\pi_1(M_1)$  is an amalgamated free product $\pi_1(M_1)= \pi_1(M_+) \ast_{\pi_1(R_1)} \pi_1(M_-)$.
If $F$ is non-separating $\pi_1(M_1)$ is an HNN extension $\pi_1(M_1)= \pi_1(M_F) \ast_{\pi_1(R_1)}$.
Thus by definition of $r(F)$ it follows that $r(\pi_1(M_1)) \leq r(F \subset M)$.

Before proceeding further we give a more precise description and some terminology for the HNN extension case. We fix a base point $\ast$ on $M_F$ and take a base point $\ast_{R_1}$ on $R_1$.
Take two paths $\gamma_\pm$ in $M_F$ connecting $\ast$ and $(\ast_{R_{1}})_{\pm} = \ast_{R_1} \times\{ \pm 1\}$.
We put
\[ a_1= [\gamma_+ \cdot \overline{\gamma_-}] \in \pi_1(M_1). \]
We call an element $a_1$ \emph{a generator} corresponding to the region $R_1$.
Here for paths $\alpha$ and $\beta$, we use the following notations.
\begin{itemize}
\item $\alpha\cdot \beta$ is the concatenation of paths. (Here we use the convention $\alpha$ first then $\beta$.)
\item $\overline{\alpha}$ is the reverse of the path $\alpha$.
\item $[\alpha]$ represents the element of the fundamental group represented by the loop $\alpha$.
\end{itemize}
Then
\[ \pi_1(M_1,\ast) = \langle \pi_1(M_F,\ast), a_1 \: | \: [\gamma_+ \cdot i_+(g)\cdot \overline{\gamma_+}] = a_1 [\gamma_-\cdot i_-(g) \cdot \overline{\gamma_-}]a_1^{-1} \quad ( [g] \in \pi_1(R_1,\ast_{R_1})) \rangle \]
where $i_\pm: R_{1} \rightarrow (R_1)_{\pm} \subset F_{\pm}$ is the inclusion map and $i_{\pm}(\gamma) = i_{\pm} \circ \gamma$.
Although the generator $a_1$ depends on a choice of paths $\gamma_{\pm}$, they are not important. If we take other paths $\gamma'_{\pm}$ and use the corresponding generator $a'_1 = [\gamma'_+ \cdot \overline{\gamma'_-}]$, then 
\begin{equation}
\label{eqn:choice} a_1 = (\gamma_+ \cdot \overline{\gamma'_{+}})\cdot (\gamma'_+ \cdot \overline{\gamma'_-})\cdot (\gamma'_- \cdot \overline{\gamma_{-}}) = d_+ a'_1 d_- 
\end{equation}
where $d_{+} =[\gamma_+ \cdot \overline{\gamma'_{+}}], d_- = [\gamma'_- \cdot \overline{\gamma_{-}}] \in \pi_1(M_F)$.

Let us return to the proof of the theorem.
Let $M_X$ be the 3-manifold obtained from $M_1$ by gluing $(R_i)_{\pm}$ for each region $R_i$ for $i=2,\ldots, x(D)$. Then as in the case of $M_1$ with a non-separating $F$ situation, $\pi_1(M_X)$ is an iterated HNN extension of $\pi_1(M_1)$. In particular, $\pi_1(M_X)$ is generated by
\[ \pi_1(M_1), a_2,\ldots, a_{x(D)} \] 
where $a_i$ is a generator corresponding to the region $R_i$. Hence
\[ r(\pi_1(M_X)) \leq r(\pi_1(M_1)) + (x(D)-1) \leq r(F \subset M) + x(D)-1. \]

Similarly, let $M_D$ be the 3-manifold obtained from $M_X$ by gluing the remaining regions. Thus $M_D$ is nothing but the 3-manifold obtained from $M_F$ by gluing $F'_{\pm}$. Thus $\pi_1(M_D)$ is generated by 
\begin{equation}
\label{eqn:surj} \pi_1(M_1), a_2,\ldots, a_{x(D)}, b_1,\ldots, b_{R(D)-x(D)-1}
\end{equation}
where $b_i$ is a generator that corresponds to a region in $\mathcal{R}(D) \setminus X$.
 
The link complement $E(L)$ is obtained from $M_{D}$ by gluing a suitable tangle along each neighborhood of crossing points. For each crossing $c$, we take a path $\gamma_{\pm}$ in $M_X$ from the base point $\ast$ to $c \times \{\pm 1\}$. Then we define $d_i \in \pi_1(M_D)$ as an element represented by a loop obtained by slightly perturbing the loop $\gamma_{+}\cdot(\{c\}\times[1,-1]) \cdot \overline{\gamma_{-}}$ so that it passes four corners around $c$ (see Figure \ref{fig:Dehn}).
Then gluing a tangle induces a relation $d_1d_2^{-1}=d_3d_4^{-1}$.
This is nothing but the relation of the famous Dehn presentation of the classical knot groups.
 
By definition of $d_i$, $d_i$ is a generator corresponding to the region $R$ that contains the corner, although it may not be identical to the generators $a_i, b_j$ in \eqref{eqn:surj}.

\begin{figure}[htbp]
\begin{center}
\includegraphics*[width=60mm]{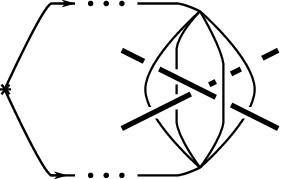}
\begin{picture}(0,0)
\put(-170,90) {$\gamma_+$}
\put(-170,20) {$\gamma_-$}
\put(-100,50) {$d_1$}
\put(-62,70) {$d_2$}
\put(-52,30) {$d_3$}
\put(-16,50) {$d_4$}
\end{picture}
\caption{Knot group relation $d_1d_2^{-1}=d_3d_4^{-1}$} 
\label{fig:Dehn}
\end{center}
\end{figure} 

Let $i_*: \pi_1(M_X) \rightarrow G(L)$
be the natural map that is induced from the inclusion map $i:M_X \rightarrow E(L)$.

\begin{claim}
\label{claim:key}
Assume that a region $R$ of a diagram is derived from a subset $X$ of $\mathcal{R}(D)$. Then the generator $b_R$ of $G(L)$ that corresponds to a region $R$ belongs to $i_*(\pi_1(M_X))$.
\end{claim}
\begin{proof}
Since $R$ is derived from $X$, by Definition \ref{definition:derived}, there exists a crossing $c$ of $D$ such that among its four corners $x_1,x_2,x_3,x_4$, three of them (say $x_1,x_2,x_3$) belong to regions in $X$ and the remaining one (say $x_4$) belongs to $R$. This implies that $d_1,d_2,d_3 \in i_*(\pi_1(M_X))$. Hence by the relation $d_1d_2^{-1}=d_3d_4^{-1}$ it follows that $d_4 \in i_*(\pi_1(M_X))$. 
Although the element $b_R$ which we take as a generator of $\pi_1(M_D)$ in \eqref{eqn:surj} may be different from $d_4$, as we have seen in \eqref{eqn:choice}, we may write $b_R = d_+ d_4 d_-$ for some $d_+, d_- \in i_*(\pi_1(M_X))$. Thus $b_R \in i_*(\pi_1(M_X))$.
\end{proof}

Since $X$ is a generating set, by Claim \ref{claim:key}, the remaining generators $b_1,\ldots,b_{R(D)-x(D)-1}$ belong to $i_*(\pi_1(M_X))$ so the map $i_*$ is a surjective homomorphism. Therefore we conclude $r(G(L)) \leq r(\pi_1(M_X)) \leq r(F \subset M)+x(D)-1$ as desired.  

\end{proof}

\subsection{Regularly cellular diagrams and a finer estimate of the generating number}

Although the trivial bound of the generating number $x(D)$ given in Lemma \ref{lemma:region-crossing} is sharp, $x(D)$ is usually much smaller than $c(D)$. Here we give a mild improvement under additional assumptions.

Let $D$ be an $F$-diagram satisfying the following.
\begin{itemize}
\item Every connected component of $D$ has at least one crossing, and,
\item Every region of $D$ is an open disk, or, an annulus which is a collar neighborhood of a connected component of $\partial F$.
\end{itemize}

Let $\widehat{F}$ be the closed surface obtained by capping off boundaries of $F$ (so if $F$ is closed, $F=\widehat{F}$). Such a diagram is often called \emph{cellular} since it induces a cellular decomposition of $\widehat{F}$, where the crossings of $D$ are 0-cells, the connected components of $D \setminus \{\text{crossings of }D\}$ are 1-cells, and the regions of $D$ are 2-cells.

\begin{definition}
We say that an $F$-diagram $D$ is \emph{regularly cellular} if it is cellular and it induces a regular cellular decomposition of $\widehat{F}$.
\end{definition}

In the classical case, for a diagram $D$ on $S^{2}$, $D$ is regularly cellular means that $D$ is connected and reduced with at least one crossing.

For a regularly cellular diagram, the closure $\overline{R}$ of a region $R$ in $F$ is a closed disk embedded in $F$. In particular, $D$ has the following property which will be crucial in our argument.

\begin{equation}
\label{eqn:diff}
\text{For each crossing} c \text{ of } D\text{, its four corners belong to different regions.}
\end{equation}

\begin{proposition}
\label{prop:generating-number}
If an $F$-diagram $D$ is regularly cellular, $x(D) \leq \frac{1}{2}R(D) + 1$ holds.
\end{proposition}
\begin{proof}
$D$ contains at least one crossing $c$.
Let $y_1,y_2,y_3,y_4$ be the corners around $c$ and let $R_1,R_2,R_3,R_4$ be the regions that contain $y_1,y_2,y_3,y_4$, respectively. By \eqref{eqn:diff} the four regions $R_1,R_2,R_3,R_4$ are mutually different. 
We take $X_{0}=\{R_1,R_2\}$ so that $R_1$ and $R_2$ share an edge.

For $i=1,\ldots,$ we inductively construct a set $X_i \subset \mathcal{R}(D)$ having the property that
\begin{equation}
\label{eqn:X_i}
\# X_i \leq 2+i, \quad \# \mathcal{D}^{\infty}(X_i) \geq 2+2i 
\end{equation}
until we arrive at a generating set.

Assume that we have constructed a set of regions $X_{i}$.

\begin{claim}
\label{claim:corner}
If $X_i$ is not generating, there exists a crossing $c$ such that among its four corners $y_1,y_2,y_3,y_4$ around $c$, exactly two corners belong to $D^{\infty}(X_i)$.
\end{claim}
\begin{proof} 
Assume to the contrary that no such crossing $c$ exists. Then there exists a region $R$ in $\mathcal{D}^{\infty}(X_i)$ and a crossing $c$ such that among the four corners around $c$, exactly one corner belongs to  $R$ and the other corners do not belong to $\mathcal{D}^{\infty}(X_i)$. 

Let $c'$ be a crossing adjacent to $c$ by a 1-cell $e$. By the regularly cellular assumption, $c'\neq c$. We show that the crossing $c'$ also has the same property, i.e., the number of corners around $c'$ that belong to $\mathcal{D}^{\infty}(X_i)$ is one.

The number of corners around $c'$ that belong to regions in $\mathcal{D}^{\infty}(X_i)$ is not two by assumption. It cannot be equal to three either, because otherwise the remaining region belongs to $\mathcal{D}(\mathcal{D}^{\infty}(X_i)) = \mathcal{D}^{\infty}(X_i)$. It cannot be equal to four, because otherwise the region adjacent to the edge $e$ belongs to $\mathcal{D}^{\infty}(X_i)$ which means that there is more than one corner around $c$ that belongs to $\mathcal{D}^{\infty}(X_i)$.

Since the region $R$ is an open disk, this means that all the crossing that sits on $\partial R$ have the same property and $\mathcal{D}^{\infty}(X_i)=\{R\}$. However, this is absurd.
\end{proof}

Assume that $X_{i}$ is not generating. By Claim \ref{claim:corner} we pick a crossing $c$ so that among its four corners $y_1,y_2,y_3,y_4$, two corners, say $y_1$ and $y_2$, belong to $\mathcal{D}^{\infty}(X_i)$, and the other corners $y_3$ and $y_4$ belong to the regions $R$ and $R'$ with $R,R' \not \in \mathcal{D}^{\infty}(X_{i})$.  By \eqref{eqn:diff}, $R \neq R'$.
Let $X_{i+1}=X_{i}\cup \{R\}$. Then
\[ R' \in \mathcal{D}(\mathcal{D}^{\infty}(X_{i}) \cup R) \subset \mathcal{D}^{\infty}(X_{i+1}) \]
Hence $\# \mathcal{D}^{\infty}(X_{i+1}) \geq \#\mathcal{D}^{\infty}(X_i) +2 \geq 2+2(i+1)$.

If $X_i$ is generating, then $2+2i \leq \# \mathcal{D}^{\infty}(X_i)= R(D)$ so we conclude
\begin{align*}
 \#X_{i} & \leq 2+i \leq \frac{R(D)}{2} + 1
\end{align*}
\end{proof}

This gives the following improvement of the crossing number bound for regularly cellular diagrams. 

\begin{theorem}
For a regularly cellular $F$-diagram $D$ of a link $L$, $c(D) \geq 2(r(G(L))-r(F \subset M)) -\chi(\widehat{F})$ holds. 
\end{theorem}
\begin{proof}
Let $E$ be the number of 1-cells of the cellular decomposition of $\widehat{F}$ induced by the diagram $D$. Since each vertex is four-valent $E=2c(D)$ holds. Thus $\chi(\widehat{F})=c(D)-E + R(D) = -c(D) + R(D)$.
Therefore by Theorem \ref{theorem:main} and Proposition \ref{prop:generating-number}
\begin{align*}
r(G(L)) & \leq r(F\subset M) + x(D) + 1 \leq r(F\subset K) + \frac{1}{2}R(D) \\
& = r(F\subset M) +\frac{1}{2}(c(D)+\chi(\widehat{F})) 
\end{align*}
\end{proof}

As we have mentioned, $x(D)$ often reflects property of diagrams, and we often get much better bounds that recover classical results.

\begin{example}
For $F=S^{2}$, an $n$-bridge diagram is a diagram in Figure \ref{fig:bridge-diagram}. For an $n$-bridge diagram $D$, $x(D) \leq n+1$ since $R_0,\ldots R_{n}$ is generating. In particular, this observation and Theorem \ref{theorem:main} recover standard inequality that $r(G(L))\leq b(L)$ of the bridge index $b(L)$ of a link $L$.
\end{example}

\begin{figure}[htbp]
\begin{center}
\includegraphics*[width=35mm]{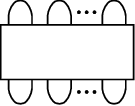}
\begin{picture}(0,0)
\put(-110,75) {$R_0$}
\put(-93,63) {$R_1$}
\put(-65,63) {$R_2$}
\put(-25,63) {$R_n$}
\put(-75,35) {\large $2n$-braid}
\end{picture}
\caption{$n$-bridge diagram and its generating regions} 
\label{fig:bridge-diagram}
\end{center}
\end{figure}

\subsection{Spatial graph}

As we mentioned in introduction, the group theoretical method can be applied to more general settings without difficulties. As a sample illustration, we briefly discuss spatial graphs.

For a spatial graph $\Gamma$ and 2-sided surface $F \subset M$, the notions of projectible and $F$-diagram are defined in an obvious manner.
We remark that we distinguish crossings of a diagram and (4-valent) vertices of the graph. To make the argument simple, we assume that $\Gamma$ is connected hence so is its $F$-diagram $D$.

The notion of generating number $x(D)$ is defined in the same way as link diagram case. Then
\[ x(D) \leq R(D) \leq b_1(\Gamma)+c(D) +1. \]

By the same argument we get the following, where $G(\Gamma)$ is the fundamental group of $M\setminus \Gamma$.
\begin{theorem}
\label{theorem:main-spatial_graph}
For a connected $F$-diagram $D$ of a spatial graph $\Gamma$ in $M$, 
\[ r(G(\Gamma)) \leq r(F \subset M)+ x(D) -1 \]
holds. In particular,
\[ c(\Gamma;F\subset M) \geq r(G(\Gamma)) -r(F \subset M) - b_1(\Gamma) \]
holds.
\end{theorem}

\section{Branched surface case}
\label{section:branch}

Let $F$ be a properly embedded, connected branched surface in $M$. This means, at each point $p$ of $F \subset M$, $F$ is locally modeled by one of the models in Figure \ref{fig:branched-surface}.
The \emph{branch locus} $B$ of $F$ is the set of points $p \in F$ whose neighborhood is modeled by (ii),(iii),(v) in Figure \ref{fig:branched-surface}. A connected component of $F-B$ is called a \emph{sector}. We say that a branched surface $F$ is \emph{2-sided} if each sector $S$ is 2-sided. We denote by $N(F)$ the regular neighborhood of $F$ which we view as $\bigcup_{S:\text{sector}} S\times[-1,1]$.
\begin{figure}[htbp]
\begin{center}
\includegraphics*[width=85mm]{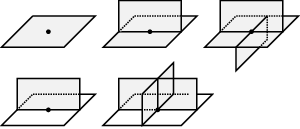}
\begin{picture}(0,0)
\put(-255,102) {(i)}
\put(-165,102) {(ii)}
\put(-85,102) {(iii)}
\put(-255,45) {(iv)}
\put(-160,45) {(v)}
\put(-152,4) {$\partial M$}
\put(-233,4) {$\partial M$}
\end{picture}
\caption{Local model of branched surface} 
\label{fig:branched-surface}
\end{center}
\end{figure}
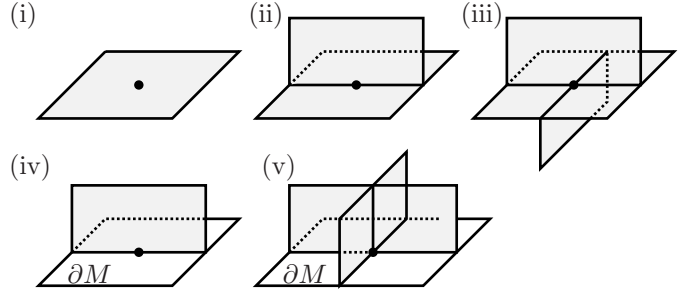 

We say that a link $L$ in $M$ is \emph{projectible} on $F$ if we can put $L \in N(F)$. For a projectible link $L$, by taking the projection on $F$ so that double point singularities (crossing points) are disjoint from a neighborhood of the branch locus $B$, we get a link diagram $D$ on $F$. The $F$-crossing number $c(L;F\subset M)$ is defined in a similar manner.

We observe the following which makes a sharp contrast with Corollary \ref{cor:unbounded}.

\begin{theorem}
\label{theorem:appendix}
For every connected 2-sided branched surface $F$ in $M$, $c(L;F\subset M)=0$ for every projectible link $L$, if the branch locus of $F$ is non-empty.
\end{theorem}
\begin{proof}
We put a diagram $D$ so that each connected component of $D$ intersects with a branch locus. Then crossing points are removed by pushing a crossing to a branch locus as shown in Figure \ref{fig:reduce_crossing}. 
\end{proof}
\begin{figure}[htbp]
\begin{center}
\includegraphics*[width=80mm]{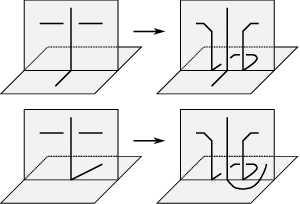}
\begin{picture}(0,0)
\end{picture}
\caption{Push-down move: Removing crossings of a diagram on a branched surface} 
\label{fig:reduce_crossing}
\end{center}
\end{figure} 
We call the operation in Figure \ref{fig:reduce_crossing} the \emph{push-down move}.

\section{Three-page index}\label{section:3-page}

Theorem \ref{theorem:appendix} shows the following.

\begin{corollary}[Every link admits a three-page presentation \cite{dy}]
\label{cor:dynnikov}
Every link in $\R^{3}$ is put on a three-page branched surface $F_{\perp}=F_1\cup F_2 \cup F_3$, where 
\[ F_1=\{(x,y,z) \: | \: z=0,x \leq 0\}, F_2=\{(x,y,z) \: | \: z=0,x \geq 0\}, F_3=\{(x,y,z)\: | \: y=0, z\geq0\}\]
\end{corollary}

Based on this observation, the three-page index of links are defined as follows.

\begin{definition}[Three-page number/index]
\label{definition:3-page}
Let $F_{\perp}$ be the three-page branched surface in Corollary \ref{cor:dynnikov}.
We call the $x$-axis $A=F_1\cap F_2 \cap F_3$ the \emph{binding} of $F_{\perp}$. For an $F_{\perp}$-diagram $D$ of a link $L$, the \emph{three-page number} $\alpha_3(D)$ is defined by 
\[ \alpha_3(D)= \#(D \cap A).\]
The \emph{three-page index}\footnote{For the unknot $K$, $\alpha_3(K)=0$ in our definition, since we do not require that $D \cap A$ consists of arcs. We adopt this convention so that we do not need to distinguish unknots in our theorem. When every component of $D$ intersects with $A$, then our $\alpha_3$ coincide with that in \cite{yo}.} $\alpha_3(L)$ of a link in $\R^{3}$ is 
\[ \alpha_3(L)=\min\{\alpha_3(D) \: | \: D \mbox{ is an }F_{\perp}\mbox{-diagram of } L \mbox{ with } c(D)=0\}\]
\end{definition}

In the following, to make the argument simple we restrict our attention to connected diagrams. In \cite{yo}, it is shown that 
\begin{equation}
\label{eqn:Yoo}
\alpha_3(L) \leq 3c(L)-1
\end{equation}
for every non-split, non-trivial link $L$ which is not the Hopf link, and asked whether this can be improved. More precisely, it is asked whether $\alpha_3(L) \leq C_1c(L) + C_2$ holds for some constant $C_1<3$ and $C_2$ \cite[Conjecture 1]{yo}.

The push-down move in Figure \ref{fig:reduce_crossing} is an operation on an $F_\perp$-diagram which decreases the crossing number $c(D)$ by one, at the cost of increasing the three-page number $\alpha_3(D)$ by two or three. Furthermore, the case where the three-page number increases by three is somewhat special so the push-down move method often allows us to find a crossingless $F_{\perp}$-diagram with three-page number smaller than the bound in \eqref{eqn:Yoo}. Indeed, it is an easy exercise to prove \eqref{eqn:Yoo} using push-down moves.

We also remark if on the diagram $D$, some over/under crossing appears consecutively then the push-down move only increase the three-page by two, possibly except the first crossing (see Figure \ref{fig:reduce_crossing_succ}). 
Thus, if a link $L$ has a minimum crossing diagram which is far from alternating, $\alpha_3(L)$ will be substantially smaller than $3c(L)-1$. 

\begin{figure}[htbp]
\begin{center}
\includegraphics*[width=80mm]{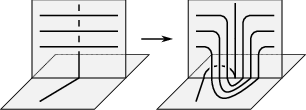}
\begin{picture}(0,0)
\end{picture}
\caption{Push-down move when under or over crossing appears consecutively.} 
\label{fig:reduce_crossing_succ}
\end{center}
\end{figure}

To give a systematic count of how many crossings can be removed by increasing the  three-page number by two for general diagrams, we use the following quantity from graph theory.

For a connected link diagram $D$ in $\R^2$, take a checkerboard coloring. The \emph{Tait graph}\footnote{Usually the edges of Tait graphs are signed according to the sign of crossings, but in our argument we do not use signs.} $\mathbb{G}(D)$ is a planar graph whose vertices are the set of black-colored regions and whose edges are crossings of $D$. A subset $B$ of the vertices of a graph $\mathbb{G}$ is \emph{independent} if no two elements of $B$ are connected by an edge of $\mathbb{G}$. The \emph{independence number} $\alpha(\mathbb{G})$ is the maximum cardinality of independent sets of $\mathbb{G}$.

\begin{theorem}
\label{theorem:3-page}
Every connected diagram $D$ of a link $L$ in $\R^{3}$ can be converted to an $F_{\perp}$-diagram $D^{*}$ of $L$ with $c(D^{*})=0$ and 
\[ \alpha_3(D^*) \leq 3c(D) + 1 - \alpha(\mathbb{G}(D)).\]
\end{theorem}

\begin{proof}
Recall that $F_{\perp}$ is the union of half-planes $F_1,F_2,F_3$ whose intersection is the axis $A$. We call $F_3 = \{(x,y,z)\: | \: y=0, z\geq0\}$ \emph{the vertical half plane}.

We assume that $D$ has at least two crossings because otherwise the assertion is obvious. 

When $D$ is non-reduced, we remove a reducible crossing. Since removing one reducible crossing reduces $c(D)$ by one and $\alpha(\mathbb{G}(D))$ by at most one,
it is sufficient to prove the assertion for reduced $D$.

Take an independent set $B$ of the Tait graph $\mathbb{G}(D)$ with $\#B = \alpha(\mathbb{G}(D))$. Thus when two regions $R$ and $R'$ belong to $B$, they share no corners or edges, that is, the closures of $R$ and $R'$ are disjoint.

We put the diagram $D$ as a diagram in the vertical half plane $F_3=\{(x,y,z)\: | \: y=0, z\geq 0\}$ of $F_{\perp}$. By picking one crossing of $D$ we get an $F_{\perp}$-diagram $D'$ with $\alpha_3(D')=4$, $c(D')=c(D)-1$ as in Figure \ref{fig:3-page_1}.

\begin{figure}[htbp]
\begin{center}
\includegraphics*[width=80mm]{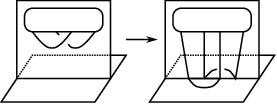}
\begin{picture}(0,0)
\put(-205,26) {$F_3$}
\put(-230,20) {$A$}
\end{picture}
\caption{$F_{\perp}$-diagram $D'$} 
\label{fig:3-page_1}
\end{center}
\end{figure}
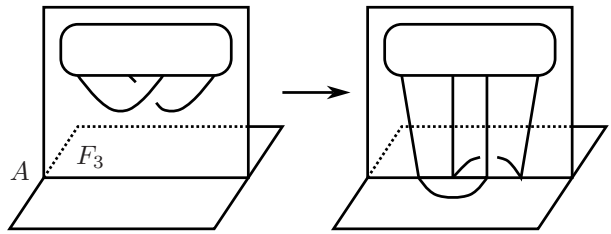 

In the following, by a region of an $F_{\perp}$-diagram $D$, we mean the regions of $D \cap F_3$, the connected components of $F_3 \setminus (D \cap F_3)$.

We apply the push-down moves until some region $R$ in $B$ shares the boundary with the binding $A$ (see Figure \ref{fig:3-page_2}). A key feature is that the $A \cap R$ is a single subarc of $A$ which is disjoint from the diagram $D$. Furthermore, no other region $R$ in $B$ intersects with $A$, and the disjointness property of regions in the independent set $B$ is preserved.

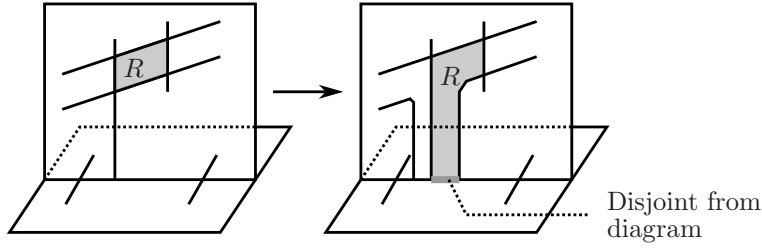
\begin{figure}[htbp]
\begin{center}
\hspace{-100pt}
\includegraphics*[width=80mm]{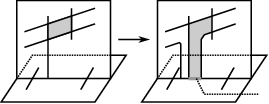}
\begin{picture}(0,0)
\put(-5,10) {Disjoint from}
\put(-5,0){diagram}
\put(-187,60) {$R$}
\put(-68,57) {$R$}
\end{picture}
\caption{Region which is a member of independent set. (We omit over-under information which is irrelevant).} 
\label{fig:3-page_2}
\end{center}
\end{figure}

Then by push-down moves we remove the crossings that appear as the corners of the region $R$ until $R$ contains one crossing. During this procedure, the property that $A \cap R$ is a single arc disjoint from the diagram is preserved (see Figure \ref{fig:3-page_3}). Furthermore, since we are focusing on regions that form an independent set, the other regions in the independent set $B$ remain disjoint from the axis $A$.

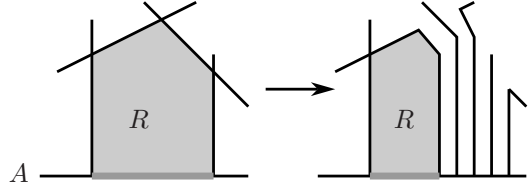
\begin{figure}[htbp]
\begin{center}
\includegraphics*[width=65mm]{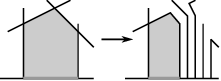}
\begin{picture}(0,0)
\put(-155,20) {$R$}
\put(-200,0) {$A$}
\put(-55,20) {$R$}
\end{picture}
\caption{Removing crossings appearing as corners of the region $R$ until it has a single corner. The intersection $A\cap R$ (bold gray arc) is a single arc disjoint from the diagram. (We omit over-under information which is irrelevant).} 
\label{fig:3-page_3}
\end{center}
\end{figure} 

Finally, the last crossing point can be removed by increasing the three-page number by two, thanks to the property that $A \cap B$ is disjoint from other part of the diagram (see Figure \ref{fig:3-page_4}). In particular, this eliminates the region $R$ from the diagram.

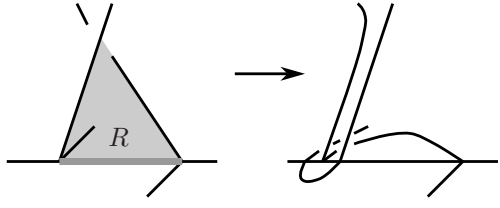
\begin{figure}[htbp]
\begin{center}
\includegraphics*[width=65mm]{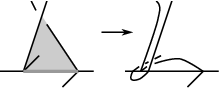}
\begin{picture}(0,0)
\put(-150,20) {$R$}
\end{picture}
\caption{The last crossing can be removed by increasing three-page number by two, even though a push-down move forces the three-page number to increase by three.} 
\label{fig:3-page_4}
\end{center}
\end{figure} 
After removing all the crossings in $R$, we are in the situation where the remaining regions in $B$ are disjoint from the axis $A$. Then we repeat these procedures until we remove all the crossings.

This argument shows that, among $c(D)-1$ crossings of $D'$, at least $\alpha(\mathbb{G}(D))$ crossings can be removed by increasing the three-page number by only two.
Therefore we get an $F$-diagram $D^{*}$ whose three-page number is at most 
\[ \alpha_3(D^{*}) \leq 4+ 2\alpha(\mathbb{G}(D)) + 3 (c(D)-1-\alpha(\mathbb{G}(D))) = 3c(D) +1  -\alpha(\mathbb{G}(D)).\]
\end{proof}

\begin{proof}[Proof of Theorem \ref{theorem:3-page-inequality}]
Let $V(\mathbb{G})$ be the set of vertices of $\mathbb{G}$. From the four-color theorem \cite{ah,rsst}, it follows that for a planar graph $\mathbb{G}$, 
\[ \frac{1}{4}\#V(\mathbb{G})\leq \alpha(\mathbb{G}) \]
holds.
Take a diagram $D$ so that $c(D)=c(L)$. Since the number of regions of $D$ is $c(L)+2$, by suitably choosing a checkerboard coloring, 
\[ 1+ \left\lceil \frac{c(L)}{2} \right\rceil \leq \#V(\mathbb{G}(D)).\]
Therefore by Theorem \ref{theorem:3-page} we conclude
\[ \alpha_3(L) \leq 3c(L) +1 - \alpha(\mathbb{G}(D)) \leq \frac{23}{8}c(L) + \frac{3}{4}.\]
\end{proof}

\section*{Acknowledgement}
The author is partially supported by JSPS KAKENHI Grant Numbers 23K03110 and 25H00588.

\end{document}